\documentclass[11pt]{article}

\usepackage[a4paper,margin=1in]{geometry}
\usepackage{amsmath,amssymb,amsthm,mathtools}
\usepackage{microtype}
\usepackage{enumitem}
\usepackage{booktabs}
\usepackage{graphicx}
\usepackage{aliascnt}
\usepackage{xcolor}
\usepackage[colorlinks=true,linkcolor=black,citecolor=blue!55!black,urlcolor=blue!55!black]{hyperref}
\usepackage[nameinlink,capitalise,noabbrev]{cleveref}
\usepackage{tikz}
\usetikzlibrary{arrows.meta,calc}

\newtheorem{theorem}{Theorem}[section]
\newaliascnt{lemma}{theorem}
\newtheorem{lemma}[lemma]{Lemma}
\aliascntresetthe{lemma}
\newaliascnt{proposition}{theorem}

\aliascntresetthe{proposition}
\newaliascnt{corollary}{theorem}
\newtheorem{corollary}[corollary]{Corollary}
\aliascntresetthe{corollary}
\newaliascnt{claim}{theorem}

\aliascntresetthe{claim}
\newaliascnt{problem}{theorem}
\newtheorem{problem}[problem]{Problem}
\aliascntresetthe{problem}
\theoremstyle{definition}
\newaliascnt{definition}{theorem}

\aliascntresetthe{definition}
\newaliascnt{construction}{theorem}

\aliascntresetthe{construction}
\theoremstyle{remark}

\newcommand{\N}{{\mathbb{N}}}

\newcommand{\R}{\mathbb{R}}

\renewcommand{\epsilon}{\varepsilon}

\renewcommand{\phi}{\varphi}
\renewcommand{\leq}{\leqslant}

\numberwithin{equation}{section}
\newcommand{\cL}{\mathcal L}

\newcommand{\cW}{\mathcal W}
\newcommand{\cWz}{\mathcal W_0}
\newcommand{\Ran}{\operatorname{Ran}}
\newcommand{\rank}{\operatorname{rank}}
\newcommand{\Span}{\operatorname{span}}

\newcommand{\one}{\mathbf 1}
\newcommand{\ip}[2]{\left\langle #1,#2\right\rangle}

\title{Infinitesimal finite forcibility and step kernels}
\author{Xichao Shu\thanks{Institute of Mathematics, Leipzig University, Leipzig, Germany. Email: \texttt{xichao.shu@uni-leipzig.de}.}\and 
Jing Yu\thanks{Shanghai Center for Mathematical Sciences, Fudan University, Shanghai, China. 
E-mail: \texttt{\{jyu@fudan.edu.cn, jczhang24@m.fudan.edu.cn\} }. } \and Junchi Zhang\footnotemark[2]}

\date{}

\begin{document}
\maketitle

\begingroup
\renewcommand{\thefootnote}{\fnsymbol{footnote}}
\footnotetext[3]{%
Xichao Shu was supported by the Alexander von Humboldt Foundation
in the framework of the Alexander von Humboldt Professorship of
Daniel Kr\'{a}l', endowed by the German Federal Ministry of Education
and Research.
Jing Yu was partially supported by the National Natural Science
Foundation of China under grants 12371343 and 12525110 (PI: Hehui Wu).
}
\endgroup
\begin{abstract}
We characterize infinitesimal finite forcibility for bounded symmetric real kernels. We prove that the graph-density gradients at a kernel span a finite-dimensional space if and only if the kernel is a step kernel. Combined with known finite-forcing results for step kernels, this gives a positive answer to a question of Lovász and Szegedy on whether every infinitesimally finitely forcible kernel is finitely forcible. The proof combines spectral methods with a compression argument based on book graphs.
\end{abstract}



\section{Introduction}

A central theme in combinatorial limit theory is to represent large
finite discrete structures by analytic limit objects while retaining
the asymptotic densities of fixed finite patterns.  For dense graphs,
the corresponding limit objects are graphons, namely symmetric
measurable functions $W\colon[0,1]^2\to[0,1]$; see
\cite{LS06,Lovasz12}.  More generally, we call a bounded measurable symmetric function
$W\colon[0,1]^2\to\R$ a \emph{kernel}. We write $\cW$ for the real vector space of kernels and $\cWz\subseteq\cW$ for the space of graphons. 

For a finite simple graph $F$ and $W\in\cW$, the
\emph{homomorphism density} of $F$ in $W$ is
\[
  t(F,W)
  :=
  \int_{[0,1]^{V(F)}}
  \prod_{\{i,j\}\in E(F)} W(x_i,x_j)
  \prod_{i\in V(F)} dx_i.
\]
For graphons, these quantities encode the asymptotic densities of
finite patterns in convergent dense graph sequences
\cite{LS06,Lovasz12}.  We say that two kernels $U,W\in\cW$ are
\emph{weakly isomorphic} if $t(F,U)=t(F,W)$ for every finite
simple graph $F$. Thus weakly isomorphic kernels are
indistinguishable by all finite graph densities.  In particular,
for graphons, a measure-preserving relabelling of the underlying
space $[0,1]$ leaves every homomorphism density unchanged and hence
preserves the weak isomorphism class
\cite{BCL10,Lovasz12}.

This naturally leads to a finite determination problem: when can
the weak isomorphism class of a kernel be determined by only
finitely many graph densities? Lovász and Szegedy introduced the notion of \emph{finite
forcibility} to study this question~\cite{LS11}. We say that $W\in\cW$ is \emph{finitely forcible in $\cW$} if there are finite simple graphs $F_1,\ldots,F_m$ such that every $U\in\cW$ satisfying
\[
        t(F_i,U)=t(F_i,W)
        \qquad\text{for all }i\in[m]
\]
is weakly isomorphic to $W$. Requiring $U,W\in\cWz$ gives the usual notion of finite forcibility for graphons.

A classical example comes from quasirandom graphs. The theorem of
Chung, Graham, and Wilson~\cite{CGW89} implies that the constant
graphon with value $p$ is determined by its edge density and
$4$-cycle density.  More generally, step graphons form a natural
class of finitely forcible objects.  A kernel $W$ is a
\emph{step kernel} if there is a finite measurable partition
\[
[0,1]=S_1\sqcup\cdots\sqcup S_q
\]
such that $W$ is constant
almost everywhere on every $S_i\times S_j$.  Step graphons are
finitely forcible by a theorem of Lovász and Sós~\cite{LS08}; see also \cite[Theorem~5.33]{Lovasz12}.
The corresponding result for bounded real-valued step kernels was
proved by Grzesik, Král, and Pikhurko~\cite[Theorem~10]{GKP24}. The constant case is immediate:
a constant kernel is forced by $K_2$ and $C_4$.

Finite forcibility, however, does not by itself imply step-like or otherwise simple structure. Lov\'asz and Szegedy constructed finitely forcible non-step graphons \cite{LS11}. Subsequent work developed the subject in connection with permutons, compactness, and weak regularity \cite{GGKK15,GKV19,CKKN18}, and produced finitely forcible graphons with infinite-dimensional typical-vertex space \cite{GKK19}. In fact, every graphon occurs as a subgraphon of a finitely forcible graphon \cite{CKM18}, and it can even occupy an arbitrarily large proportion of one \cite{KLNS20}. Thus finite forcibility alone places very little restriction on the structural complexity of a graphon.

Lov\'asz and Szegedy proposed a different, infinitesimal finiteness condition, obtained by linearizing all graph-density functionals at a fixed kernel $W$. To formulate this condition, it is convenient to work in $\cW$ rather than only in $\cWz$, since $W$ can then be perturbed in an arbitrary bounded symmetric direction. For every finite simple graph $F$ and $W,K\in\cW$, the directional derivative of the polynomial functional $t(F,\cdot)$ at $W$ can be written as
\[
        \left.
        \frac{d}{d\varepsilon}t(F,W+\varepsilon K)
        \right|_{\varepsilon=0}
        =
        \int_{[0,1]^2}K(x,y)\nabla F(W)(x,y)\,dx\,dy.
\]
The symmetric kernel $\nabla F(W)$ is the \emph{graph-density gradient}; its explicit edge-deletion formula is given in \cref{sec:preliminaries}. Define
\[
        \cL(W)
        :=
        \Span\{\nabla F(W):F\text{ is a finite simple graph}\}.
\]
Following \cite[Remark~7.8]{LS11}, we call $W$ \emph{infinitesimally finitely forcible} if $\dim\cL(W)<\infty$. Concretely, this means that there are finitely many graphs $G_1,\ldots,G_s$ such that, for every simple graph $F$,  $\nabla F(W)$ is a linear combination of $\nabla G_1(W),\ldots,\nabla G_s(W)$.

The two notions of finiteness concern different mathematical data. Finite forcibility asks whether finitely many exact, nonlinear density equations determine $W$ globally up to weak isomorphism. Infinitesimal finite forcibility instead linearises \emph{all} graph-density functionals at one kernel and asks whether their differentials span a finite-dimensional space.
Lov\'asz and Szegedy proved that, among finitely forcible graphons, the latter condition is equivalent to being a step graphon \cite[Proposition~7.7 and Remark~7.8]{LS11}. Their finitely forcible non-step examples therefore show that finite forcibility does not imply infinitesimal finite forcibility. They asked whether the converse holds:

\begin{problem}[{\cite[Question~8]{LS11}}]\label{pro:main}
Is every infinitesimally finitely forcible graphon also finitely
forcible, and hence a step graphon?
\end{problem}

We answer this question affirmatively. In fact, we remove the finite-forcibility assumption from the preceding characterization and prove the stronger statement for all bounded real-valued kernels.

\begin{theorem}\label{thm:main}
Let $W\in\cW$. Then
\[
        \dim\cL(W)<\infty
\] if and only if it is a step kernel.
\end{theorem}

Combining \cref{thm:main} with the finite forcibility of step kernels gives the following direct answer to the question of Lov\'asz and Szegedy.

\begin{corollary}\label{cor:forcing}
Every infinitesimally finitely forcible kernel is finitely forcible in $\cW$. In particular, every infinitesimally finitely forcible graphon is finitely forcible.
\end{corollary}

The proof of \cref{thm:main} has two main ingredients. The easy direction follows because all graph-density gradients of a $q$-step kernel are constant on the same product partition and hence lie in a space of dimension at most $q(q+1)/2$. For the converse, cycle gradients show that the integral operator with kernel $W$ has finite rank. This is not sufficient by itself: for example, $W(x,y)=xy$ has rank one but is not a step kernel. The new ingredient is a compression identity for book graphs. It implies that all monomials in the finitely many spectral coordinates of $W$ lie in a common finite-dimensional subspace of $L^2[0,1]$. Hence each spectral coordinate satisfies a nonzero polynomial relation and has finite essential range; their joint level sets then give a finite step partition. Thus cycles force $T_W$ to have finite rank, while books force the resulting spectral coordinate functions to have finite essential range.

\cref{sec:preliminaries} gives the explicit gradient formula and fixes the operator notation. The proof of \cref{thm:main} is given in \cref{sec:proof}.

\section{Gradients and operator notation}\label{sec:preliminaries}
We give the explicit edge-deletion formula for the gradient introduced above and fix the functional-analytic notation used in the proof. All Hilbert spaces and inner products below are real.  We use
\[
 \ip{f}{g}:=\int_0^1 f(x)g(x)\,dx
\]
for the inner product on $L^2[0,1]$, and the analogous notation on
$L^2([0,1]^2)$.
For $m\ge1$, put $[m]:=\{1,\ldots,m\}$; we write
$\N:=\{1,2,\ldots\}$ and $\N_0:=\{0,1,2,\ldots\}$.

Let $H$ be a graph with two distinct vertices labelled $1$ and $2$.  Its \emph{two-labelled density} in $W$ is the function
\begin{equation}\label{eq:t2}
 t_2(H,W)(x,y)
 :=
 \int_{[0,1]^{V(H)\setminus\{1,2\}}}
 \prod_{\{i,j\}\in E(H)}W(z_i,z_j)
 \prod_{i\in V(H)\setminus\{1,2\}}dz_i,
\end{equation}
where $z_1=x$ and $z_2=y$.

For a finite simple graph $F$, let
\[
 \vec E(F):=
 \bigcup_{\{a,b\}\in E(F)}\{(a,b),(b,a)\}.
\]
For $(a,b)\in\vec E(F)$, let $F_{a,b}$ be obtained by deleting
$\{a,b\}$ and labelling $a$ by $1$ and $b$ by $2$.  Then the \emph{gradient} is given explicitly by
\begin{equation}\label{eq:gradient}
 \nabla F(W)
 :=
 \frac12\sum_{(a,b)\in\vec E(F)}t_2(F_{a,b},W).
\end{equation}
The two orientations make this kernel symmetric.  With this normalization, we have the first-variation formula of Lovász and Szegedy.

\begin{lemma}[Lovász--Szegedy {\cite[Lemma~7.5]{LS11}}]
\label{lem:first-variation}
Let $F$ be a finite simple graph and let $W,K\in\cW$.  Then
\[
 \left.\frac{d}{d\varepsilon}t(F,W+\varepsilon K)
 \right|_{\varepsilon=0}
 =
 \ip{K}{\nabla F(W)}_{L^2([0,1]^2)}.
\]
\end{lemma}

For $K\in L^2([0,1]^2)$, not necessarily symmetric, let $T_K$ be
the \emph{Hilbert--Schmidt operator} on $L^2[0,1]$ defined by
\[
  (T_Kf)(x):=\int_0^1 K(x,y)f(y)\,dy.
\]
The map $K\mapsto T_K$ is injective and
$\|T_K\|_{\mathrm{HS}}=\|K\|_2$.  Conversely, every
Hilbert--Schmidt operator on $L^2[0,1]$ has a unique kernel in
$L^2([0,1]^2)$, which we call its Hilbert--Schmidt kernel.

For a bounded operator $S$ on $L^2[0,1]$, we write $\Ran S$ for
its range and $\rank S:=\dim\Ran S$.  If
$K^{\mathsf T}(x,y):=K(y,x)$, then
$T_{K^{\mathsf T}}=(T_K)^*$.  Every $T_K$ is compact, and it is
self-adjoint when $K$ is symmetric.  We will use the standard
spectral theorem for compact self-adjoint operators; in particular,
their nonzero eigenvalues are real and have finite multiplicity.
For $n\in\N$, we use
\[
 K^{\odot n}(x,y):=K(x,y)^n
\]
for the pointwise power; this is different from the operator-composition
power $(T_K)^n$.

For $f,g\in L^2[0,1]$, the rank-one operator $f\otimes g$ is
\[
 (f\otimes g)h:=\ip{h}{g}f.
\]
We write $\one$ for the constant function $\one(x)=1$.  If $V$ is a
closed subspace of $L^2[0,1]$, the orthogonal projections onto $V$ and
$V^\perp$ are denoted by $P_V$ and $Q_V$, respectively; $I$ denotes the
identity operator.

\section{Proof of \cref{thm:main}}\label{sec:proof}

In this section, we prove Theorem~\ref{thm:main}.  We first consider the easy direction.  
Suppose that $W$ is a $q$-step kernel with respect to a measurable partition $[0,1]=S_1\sqcup\cdots\sqcup S_q$.  For every two-labelled graph $H$, the function $t_2(H,W)$ is constant almost everywhere on each $S_i\times S_j$, since one may split the defining integral according to the parts containing the unlabelled vertices.  Consequently, every $\nabla F(W)$ belongs to the space of symmetric kernels that are constant on the rectangles $S_i\times S_j$.  
This space has dimension at most $q(q+1)/2$, and hence $\dim\cL(W)<\infty$.

It remains to prove the converse, which is the main part of this section.  
Assume that $\dim\cL(W)<\infty$.  Cycle gradients first
imply that $T_W$ has finite rank.  We then use book graphs and a compression argument to control the pointwise powers of the kernel of $T_W^2$.  
Combined with the spectral decomposition of $W$, this forces all spectral monomials to lie in a common finite-dimensional function space.  
It follows that the spectral coordinates have finite essential range, which yields a finite step partition for $W$.

We will use the following consequence of the cycle-gradient argument of Lovász and Szegedy~\cite[Example~7.4 and the proof of Proposition~7.7]{LS11}.  For completeness, we recall the short argument.

\begin{lemma}[Cycle gradients]\label{lem:cycles}
Let $W\in\cW$ and put $T=T_W$.  
For $m\ge2$, let $K_m$ be the
Hilbert--Schmidt kernel of $T^m$.  Then
\[
  \nabla C_{m+1}(W)=(m+1)K_m.
\]
Consequently, if $\cL(W)$ is finite-dimensional, then $T$ has finite rank.
\end{lemma}

\begin{proof}
Deleting an edge from $C_{m+1}$ leaves a two-labelled path of length $m$, whose density is the kernel $K_m$.  Since all $m+1$ edges give the same contribution, the displayed identity follows.

If $\cL(W)$ is finite-dimensional, then $T^2,T^3,\ldots$ are
linearly dependent, so $p(T)=0$ for some nonzero polynomial $p$.
Hence every nonzero eigenvalue of the compact self-adjoint
operator $T$ is a nonzero root of $p$.  There are only finitely
many such roots, each with finite multiplicity, and thus $T$ has
finite rank.
\end{proof}

From now on, assume that $\dim\cL(W)<\infty$.  By
Lemma~\ref{lem:cycles}, the operator $T:=T_W$ has finite rank.
Set $V:=\Ran T$, and let $Q$ be the orthogonal projection onto
$V^\perp$.  Let $C$ denote the kernel of $T^2$, so that
$$C(x,y):=\int_0^1 W(x,z)W(z,y)\,dz.$$

For $n\ge1$, let $B_n$ denote the \emph{$n$-book}, consisting of
$n$ triangles sharing a common edge.  More precisely, its vertices
are $a,b,z_1,\ldots,z_n$, with edges $ab$, $az_i$, and $bz_i$ for
$i\in[n]$.  In particular, $B_1=C_3$.

The following compression identity is the key point of the proof.

\begin{lemma}[Book compression]\label{lem:book}
For every $n\ge1$,
\[
  QT_{\nabla B_n(W)}Q
  =
  QT_{C^{\odot n}}Q.
\]
Consequently, the operators
$A_n:=QT_{C^{\odot n}}Q$, $n\ge1$, span a finite-dimensional
vector space.
\end{lemma}

\begin{proof}
Fix $n\ge1$.  Since $W$ is bounded, the kernel $C$ is bounded and
symmetric.  Define
\[
  R_n(x,z)
  :=
  \int_0^1
  W(x,y)W(z,y)C(x,y)^{n-1}\,dy,
\]
where, when $n=1$, we interpret $C(x,y)^{n-1}$ as $1$.

We first compute the gradient of $B_n$.  Recall that
$\nabla B_n(W)$ is obtained by deleting each edge of $B_n$ in turn,
labelling its two endpoints, and summing the corresponding
two-labelled densities over the two orientations.
Consider first the common edge $\{a,b\}$.  Fixing $a=x$ and $b=y$,
and deleting $\{a,b\}$, each page contributes
\[
  \int_0^1 W(x,z)W(z,y)\,dz=C(x,y).
\]
Since the $n$ page vertices are integrated independently, the
orientation $(a,b)$ contributes $C(x,y)^n=C^{\odot n}(x,y)$.
The reverse orientation gives the same kernel, since $C$ is
symmetric.  Hence, after the factor $1/2$ in the definition of
the gradient, the common edge contributes one copy of
$C^{\odot n}$.

We next consider the two noncommon edges on a fixed page, say the
page with vertex $z_i$.  Delete the edge $\{a,z_i\}$ and label
$a=x$ and $z_i=z$.  If $b=y$, then the edge $\{a,b\}$ contributes
$W(x,y)$, the remaining edge $\{b,z_i\}$ contributes $W(z,y)$,
and each of the other $n-1$ pages contributes $C(x,y)$.  Thus the
orientation $(a,z_i)$ contributes
\[
  \int_0^1
  W(x,y)W(z,y)C(x,y)^{n-1}\,dy
  =
  R_n(x,z).
\]
The reverse orientation contributes $R_n^{\mathsf T}(x,z)$.
The same two kernels arise from the two orientations of
$\{b,z_i\}$.  Consequently, after the factor $1/2$ in the
definition of $\nabla B_n(W)$, the two noncommon edges on this
page contribute $R_n+R_n^{\mathsf T}$.  Summing over the $n$
pages gives
\[
  \nabla B_n(W)
  =
  C^{\odot n}
  +
  n\bigl(R_n+R_n^{\mathsf T}\bigr).
  \tag{3.1}\label{eq:book-gradient}
\]

We now show that the last two terms disappear after compression
by $Q$.  For almost every $x$, define
$h_{n,x}(y):=W(x,y)C(x,y)^{n-1}$.  Since both $W$ and $C$ are
bounded, $h_{n,x}\in L^2[0,1]$.  By the definition of $T$,
\[
  (Th_{n,x})(z)
  =
  \int_0^1
  W(z,y)W(x,y)C(x,y)^{n-1}\,dy
  =
  R_n(x,z).
\]
Hence $R_n(x,\cdot)\in\Ran T=V$ for almost every $x$.

Let $f\in L^2[0,1]$.  Since $Qf\in V^\perp$, we obtain, for almost
every $x$,
\[
  (T_{R_n}Qf)(x)
  =
  \int_0^1 R_n(x,y)(Qf)(y)\,dy
  =
  \ip{R_n(x,\cdot)}{Qf}
  =
  0.
\]
Thus $T_{R_n}Q=0$.  Since $Q$ is self-adjoint and
$T_{R_n^{\mathsf T}}=T_{R_n}^*$, taking adjoints gives
\[
  QT_{R_n^{\mathsf T}}=0.
\]
It follows that
$QT_{R_n}Q=QT_{R_n^{\mathsf T}}Q=0$.

Applying the map $K\mapsto T_K$ to
\eqref{eq:book-gradient} and compressing on both sides by $Q$
therefore yields
\[
  QT_{\nabla B_n(W)}Q
  =
  QT_{C^{\odot n}}Q,
\]
which proves the claimed compression identity.

Finally, put $A_n:=QT_{C^{\odot n}}Q$.  By the identity just proved,
$A_n=QT_{\nabla B_n(W)}Q$ for every $n\ge1$.  Since
$\nabla B_n(W)\in\cL(W)$ and $\cL(W)$ is finite-dimensional, while
the map $K\mapsto QT_KQ$ is linear, the operators
$A_1,A_2,\ldots$ span a finite-dimensional vector space.
\end{proof}

We next pass from the compressed powers of $C$ to the spectral
coordinates of $W$.  If $T=0$, then $W=0$ almost everywhere by
the injectivity of the map $K\mapsto T_K$, and there is nothing
to prove.  We may therefore assume that $T$ has nonzero rank.

We use the following standard finite-rank form of the spectral theorem; see~\cite[Section~7.5]{Lovasz12}. 

\begin{lemma}[Finite-rank spectral decomposition]
\label{lem:spectral}
Let $W\in\cW$ and suppose that $T_W$ has finite nonzero rank $r$.
Then there are nonzero real eigenvalues
$\lambda_1,\ldots,\lambda_r$ and bounded real orthonormal
eigenfunctions $e_1,\ldots,e_r$ such that
\[
  W(x,y)
  =
  \sum_{j=1}^r \lambda_j e_j(x)e_j(y)
\]
almost everywhere.  Moreover,
$\Ran T_W=\Span\{e_1,\ldots,e_r\}$.
\end{lemma}

Fix the representation in Lemma~\ref{lem:spectral} and put
$u_j:=\lambda_j e_j$ for $j\in[r]$.  Since
$u_j = T e_j$, for almost every $x$, we have
\[
        |u_j(x)|
        =|Te_j(x)|
        \leq \|W(x,\cdot)\|_2\|e_j\|_2
        \leq \|W\|_\infty.
\]
In particular, each $u_j$ (and hence each $e_j$) is essentially bounded.
Moreover, since $C$ is the
Hilbert--Schmidt kernel of $T^2$,
\[
  C(x,y)
  =
  \sum_{j=1}^r u_j(x)u_j(y)
\]
almost everywhere.

For $\alpha=(\alpha_1,\ldots,\alpha_r)\in\N_0^r$, write
$|\alpha|:=\alpha_1+\cdots+\alpha_r$ and
$u^\alpha:=u_1^{\alpha_1}\cdots u_r^{\alpha_r}$.  In particular,
$u^{(0,\ldots,0)}=\one$.  Since the functions $u_j$ are bounded,
all these monomials belong to $L^2[0,1]$.

The next lemma is the main consequence of the book-compression
identity.

\begin{lemma}[Monomial confinement]\label{lem:monomials}
There is a finite-dimensional subspace
$\mathcal U\subseteq L^2[0,1]$ such that $\one\in\mathcal U$ and
$u^\alpha\in\mathcal U$ for every $\alpha\in\N_0^r$.
\end{lemma}

\begin{proof}
For $n\ge1$, let
$H_n:=\Span\{u^\alpha:|\alpha|=n\}$.  Since
$C(x,y)=\sum_{j=1}^r u_j(x)u_j(y)$, the multinomial theorem gives
\[
  C^{\odot n}(x,y)
  =
  \sum_{|\alpha|=n}
  \binom{n}{\alpha}
  u^\alpha(x)u^\alpha(y),
\]
where
$\binom{n}{\alpha}:=
n!/(\alpha_1!\cdots\alpha_r!)$.
Therefore
\[
  T_{C^{\odot n}}
  =
  \sum_{|\alpha|=n}
  \binom{n}{\alpha}
  u^\alpha\otimes u^\alpha.
\]
Compressing by $Q$ on both sides yields
\[
  A_n
  =
  \sum_{|\alpha|=n}
  \binom{n}{\alpha}
  (Qu^\alpha)\otimes(Qu^\alpha).
  \tag{3.2}\label{eq:gram-expansion}
\]
Indeed, for every $f,g\in L^2[0,1]$,
$Q(f\otimes g)Q=(Qf)\otimes(Qg)$.

Set
$S_n:=\Span\{Qu^\alpha:|\alpha|=n\}=Q(H_n)$.
Since all coefficients in \eqref{eq:gram-expansion} are positive,
for every $f\in L^2[0,1]$ we have
\[
  \ip{A_nf}{f}
  =
  \sum_{|\alpha|=n}
  \binom{n}{\alpha}
  \bigl|\ip{f}{Qu^\alpha}\bigr|^2.
\]
It follows that $A_nf=0$ if and only if
$f$ is orthogonal to every $Qu^\alpha$ with $|\alpha|=n$.
Hence
$\ker A_n=S_n^\perp$.

The operator $A_n$ is self-adjoint and has finite rank, since
\eqref{eq:gram-expansion} is a finite sum of rank-one operators.
Therefore its range is closed, and
\[
  \Ran A_n
  =
  (\ker A_n)^\perp
  =
  S_n
  =
  Q(H_n).
  \tag{3.3}\label{eq:range-An}
\]

By Lemma~\ref{lem:book}, the space
$\mathcal A:=\Span\{A_n:n\ge1\}$ is finite-dimensional. 
If $\mathcal A=\{0\}$, set $M:=\{0\}$. 
Otherwise, choose a basis $A_{n_1},\ldots,A_{n_s}$ of $\mathcal A$ from the
sequence $(A_n)_{n\ge1}$ and set
\[
  M
  :=
  \Ran A_{n_1}+\cdots+\Ran A_{n_s}.
\]
In either case, the space $M$ is finite-dimensional and $\Ran A_n\subseteq M$ for every $n \ge 1$.  
Together with \eqref{eq:range-An}, this gives $Q(H_n)\subseteq M$ for every $n\ge1$.

Let $P$ be the orthogonal projection onto $V$.  Since $P+Q=I$,
for every $h\in H_n$ we have $h=Ph+Qh$, where $Ph\in V$ and
$Qh\in M$.  Thus
$H_n\subseteq V+M$ for every $n\ge1$.

Finally, set
$\mathcal U:=\Span\{\one\}+V+M$.  Both $V$ and $M$ are
finite-dimensional, so $\mathcal U$ is finite-dimensional.
It contains $\one$ and every monomial $u^\alpha$ with
$|\alpha|\ge1$, and hence every $u^\alpha$ with
$\alpha\in\N_0^r$.
\end{proof}

We are now ready to complete the proof of
Theorem~\ref{thm:main}.  By Lemma~\ref{lem:monomials}, all
monomials in the spectral coordinates lie in a common
finite-dimensional space.  This immediately forces each spectral
coordinate to take only finitely many values almost everywhere.

\begin{proof}[Proof of Theorem~\ref{thm:main}]
The forward direction was proved at the beginning of this section,
so it remains to assume that $\dim\cL(W)<\infty$ and show that $W$
is a step kernel.

Put $T:=T_W$.  By Lemma~\ref{lem:cycles}, the operator $T$ has
finite rank.  If $T=0$, then the injectivity of the map
$K\mapsto T_K$ gives $W=0$ almost everywhere, and hence $W$ is a
step kernel.

We may therefore assume that $T$ has nonzero rank.  Fix the
spectral representation introduced above and write
$u_j=\lambda_j e_j$ for $j\in[r]$.  By
Lemma~\ref{lem:monomials}, there is a finite-dimensional subspace
$\mathcal U\subseteq L^2[0,1]$ containing $\one$ and every
monomial $u^\alpha$.

Fix $j\in[r]$.  In particular, all the functions
$\one,u_j,u_j^2,\ldots$ belong to $\mathcal U$.  Since
$\mathcal U$ is finite-dimensional, these functions are linearly
dependent.  Thus there is a nonzero polynomial $p_j\in\R[X]$ such
that $p_j(u_j)=0$ in $L^2[0,1]$, and hence
$p_j(u_j(x))=0$ for almost every $x\in[0,1]$.  Since a nonzero
real polynomial has only finitely many real roots, there are a finite set $R_j\subseteq\R$ and a conull measurable set $X_j\subseteq[0,1]$ such that $u_j(X_j)\subseteq R_j$. Set $X_0:=\bigcap_{j=1}^rX_j$.
Then $X_0$ is conull, and the map $x \mapsto(u_1(x),\ldots,u_r(x))$ takes at most $\prod_{j=1}^r|R_j|$ values on $X_0$.
Let $S_1,\ldots,S_q$ be its nonempty level sets, and absorb the null set $[0,1]\setminus X_0$ into one of them.  These sets form a finite measurable partition of $[0,1]$.

Finally, since
\[
  W(x,y)
  =
  \sum_{j=1}^r
  \frac{u_j(x)u_j(y)}{\lambda_j}
\]
for almost every $(x, y) \in [0, 1]^2$. 
On the conull set on which this identity holds, its right-hand side depends only on the level sets containing $x$ and $y$.  Hence $W$ is constant almost everywhere on each $S_a\times S_b$, and therefore $W$ is a step kernel.
\end{proof}

\section*{Declaration of AI use.}
ChatGPT-5.6 sol, developed by OpenAI, was used during the early stages of this project and in the
preparation of the manuscript. In particular, the core idea underlying the proof of the main theorem was proposed by ChatGPT.
ChatGPT was also used to assist with the language,
organization, and presentation of the manuscript.  The authors subsequently developed and independently verified all mathematical
arguments and take full responsibility for all statements, proofs, citations, and conclusions presented in this paper.

\bibliographystyle{abbrv}
\bibliography{reference}

@article {BCL10,
    AUTHOR = {Borgs, Christian and Chayes, Jennifer and Lov\'asz,
              L\'aszl\'o},
     TITLE = {Moments of two-variable functions and the uniqueness of graph
              limits},
   JOURNAL = {Geom. Funct. Anal.},
  FJOURNAL = {Geometric and Functional Analysis},
    VOLUME = {19},
      YEAR = {2010},
    NUMBER = {6},
     PAGES = {1597--1619},
      ISSN = {1016-443X,1420-8970},
   MRCLASS = {05C60 (05C80 28D15 60C05)},
  MRNUMBER = {2594615},
MRREVIEWER = {Christian\ Lavault},
       DOI = {10.1007/s00039-010-0044-0},
       URL = {https://doi.org/10.1007/s00039-010-0044-0},
}

@article {CGW89,
    AUTHOR = {Chung, F. R. K. and Graham, R. L. and Wilson, R. M.},
     TITLE = {Quasi-random graphs},
   JOURNAL = {Combinatorica},
  FJOURNAL = {Combinatorica. An International Journal on Combinatorics and
              the Theory of Computing},
    VOLUME = {9},
      YEAR = {1989},
    NUMBER = {4},
     PAGES = {345--362},
      ISSN = {0209-9683},
   MRCLASS = {05C80},
  MRNUMBER = {1054011},
MRREVIEWER = {Zbigniew\ Palka},
       DOI = {10.1007/BF02125347},
       URL = {https://doi.org/10.1007/BF02125347},
}

@article {CKKN18,
    AUTHOR = {Cooper, Jacob W. and Kaiser, Tom\'a\v{s} and Kr\'{a}l', Daniel
              and Noel, Jonathan A.},
     TITLE = {Weak regularity and finitely forcible graph limits},
   JOURNAL = {Trans. Amer. Math. Soc.},
  FJOURNAL = {Transactions of the American Mathematical Society},
    VOLUME = {370},
      YEAR = {2018},
    NUMBER = {6},
     PAGES = {3833--3864},
      ISSN = {0002-9947,1088-6850},
   MRCLASS = {05C35 (05C80)},
  MRNUMBER = {3811511},
MRREVIEWER = {Jonathan\ Cutler},
       DOI = {10.1090/tran/7066},
       URL = {https://doi.org/10.1090/tran/7066},
}

@article {CKM18,
    AUTHOR = {Cooper, Jacob W. and Kr\'{a}l', Daniel and Martins, Ta\'isa
              L.},
     TITLE = {Finitely forcible graph limits are universal},
   JOURNAL = {Adv. Math.},
  FJOURNAL = {Advances in Mathematics},
    VOLUME = {340},
      YEAR = {2018},
     PAGES = {819--854},
      ISSN = {0001-8708,1090-2082},
   MRCLASS = {05C35},
  MRNUMBER = {3886181},
       DOI = {10.1016/j.aim.2018.10.019},
       URL = {https://doi.org/10.1016/j.aim.2018.10.019},
}

@article {GKK19,
    AUTHOR = {Glebov, Roman and Klimo\v{s}ov\'a, Tereza and Kr\'{a}l', Daniel},
     TITLE = {Infinite-dimensional finitely forcible graphon},
   JOURNAL = {Proc. Lond. Math. Soc. (3)},
  FJOURNAL = {Proceedings of the London Mathematical Society. Third Series},
    VOLUME = {118},
      YEAR = {2019},
    NUMBER = {4},
     PAGES = {826--856},
      ISSN = {0024-6115,1460-244X},
   MRCLASS = {05C35 (05C80)},
  MRNUMBER = {3938713},
MRREVIEWER = {Jan\ Hladk\'y},
       DOI = {10.1112/plms.12203},
       URL = {https://doi.org/10.1112/plms.12203},
}

@article {GKV19,
    AUTHOR = {Glebov, Roman and Kr\'{a}l', Daniel and Volec, Jan},
     TITLE = {Compactness and finite forcibility of graphons},
   JOURNAL = {J. Eur. Math. Soc. (JEMS)},
  FJOURNAL = {Journal of the European Mathematical Society (JEMS)},
    VOLUME = {21},
      YEAR = {2019},
    NUMBER = {10},
     PAGES = {3199--3223},
      ISSN = {1435-9855,1435-9863},
   MRCLASS = {05C35 (05C82)},
  MRNUMBER = {3994104},
MRREVIEWER = {Jan\ Hladk\'y},
       DOI = {10.4171/JEMS/901},
       URL = {https://doi.org/10.4171/JEMS/901},
}

@article {GKP24,
    AUTHOR = {Grzesik, Andrzej and Kr\'{a}l', Daniel and Pikhurko, Oleg},
     TITLE = {Forcing generalised quasirandom graphs efficiently},
   JOURNAL = {Combin. Probab. Comput.},
  FJOURNAL = {Combinatorics, Probability and Computing},
    VOLUME = {33},
      YEAR = {2024},
    NUMBER = {1},
     PAGES = {16--31},
      ISSN = {0963-5483,1469-2163},
   MRCLASS = {05C40},
  MRNUMBER = {4680487},
       DOI = {10.1017/s0963548323000263},
       URL = {https://doi.org/10.1017/s0963548323000263},
}

@article {KLNS20,
    AUTHOR = {Kr\'{a}l', Daniel and Lov\'asz, L\'aszl\'o{} M. and Noel,
              Jonathan A. and Sosnovec, Jakub},
     TITLE = {Finitely forcible graphons with an almost arbitrary structure},
   JOURNAL = {Discrete Anal.},
  FJOURNAL = {Discrete Analysis},
      YEAR = {2020},
     PAGES = {Paper No. 9, 36},
      ISSN = {2397-3129},
   MRCLASS = {05C35 (05C80)},
  MRNUMBER = {4132060},
MRREVIEWER = {Nikolaos\ Fountoulakis},
       DOI = {10.19086/da},
       URL = {https://doi.org/10.19086/da},
}

@book {Lovasz12,
    AUTHOR = {Lov\'asz, L\'aszl\'o},
     TITLE = {Large networks and graph limits},
    SERIES = {American Mathematical Society Colloquium Publications},
    VOLUME = {60},
 PUBLISHER = {American Mathematical Society, Providence, RI},
      YEAR = {2012},
     PAGES = {xiv+475},
      ISBN = {978-0-8218-9085-1},
   MRCLASS = {05-02 (05C60 05C80 05C82 05D40)},
  MRNUMBER = {3012035},
MRREVIEWER = {Anant\ P.\ Godbole},
       DOI = {10.1090/coll/060},
       URL = {https://doi.org/10.1090/coll/060},
}

@article {LS06,
    AUTHOR = {Lov\'asz, L\'aszl\'o{} and Szegedy, Bal\'azs},
     TITLE = {Limits of dense graph sequences},
   JOURNAL = {J. Combin. Theory Ser. B},
  FJOURNAL = {Journal of Combinatorial Theory. Series B},
    VOLUME = {96},
      YEAR = {2006},
    NUMBER = {6},
     PAGES = {933--957},
      ISSN = {0095-8956,1096-0902},
   MRCLASS = {05C35 (05C50 05C80)},
  MRNUMBER = {2274085},
MRREVIEWER = {Yoshiharu\ Kohayakawa},
       DOI = {10.1016/j.jctb.2006.05.002},
       URL = {https://doi.org/10.1016/j.jctb.2006.05.002},
}

@article {LS11,
    AUTHOR = {Lov\'asz, L. and Szegedy, B.},
     TITLE = {Finitely forcible graphons},
   JOURNAL = {J. Combin. Theory Ser. B},
  FJOURNAL = {Journal of Combinatorial Theory. Series B},
    VOLUME = {101},
      YEAR = {2011},
    NUMBER = {5},
     PAGES = {269--301},
      ISSN = {0095-8956,1096-0902},
   MRCLASS = {05C35 (05D40)},
  MRNUMBER = {2802882},
MRREVIEWER = {Lyuben\ R.\ Mutafchiev},
       DOI = {10.1016/j.jctb.2011.03.005},
       URL = {https://doi.org/10.1016/j.jctb.2011.03.005},
}

@article {GGKK15,
    AUTHOR = {Glebov, Roman and Grzesik, Andrzej and Klimo\v{s}ov\'a, Tereza
              and Kr\'{a}l', Daniel},
     TITLE = {Finitely forcible graphons and permutons},
   JOURNAL = {J. Combin. Theory Ser. B},
  FJOURNAL = {Journal of Combinatorial Theory. Series B},
    VOLUME = {110},
      YEAR = {2015},
     PAGES = {112--135},
      ISSN = {0095-8956,1096-0902},
   MRCLASS = {05C80 (05A05)},
  MRNUMBER = {3279390},
MRREVIEWER = {Lyuben\ R.\ Mutafchiev},
       DOI = {10.1016/j.jctb.2014.07.007},
       URL = {https://doi.org/10.1016/j.jctb.2014.07.007},
}

@article {LS08,
    AUTHOR = {Lov\'asz, L\'aszl\'o{} and S\'os, Vera T.},
     TITLE = {Generalized quasirandom graphs},
   JOURNAL = {J. Combin. Theory Ser. B},
  FJOURNAL = {Journal of Combinatorial Theory. Series B},
    VOLUME = {98},
      YEAR = {2008},
    NUMBER = {1},
     PAGES = {146--163},
      ISSN = {0095-8956,1096-0902},
   MRCLASS = {05C80},
  MRNUMBER = {2368030},
MRREVIEWER = {Daniela\ K\"uhn},
       DOI = {10.1016/j.jctb.2007.06.005},
       URL = {https://doi.org/10.1016/j.jctb.2007.06.005},
}

\end{document}